\documentclass[a4paper]{article}

\usepackage[english]{babel}
\usepackage[utf8]{inputenc}
\usepackage{amsmath}
\usepackage{graphicx}
\usepackage[colorinlistoftodos]{todonotes}
\usepackage{amsmath, amssymb, verbatim, xspace,amsthm,latexsym,fullpage, dsfont, indentfirst, hyperref, enumerate}
\usepackage[all]{xy}

\newtheorem{thm}{Theorem}[section]
\newtheorem{proposition}[thm]{Proposition}
\newtheorem{definition}[thm]{Definition}

\newtheorem{lemma}[thm]{Lemma}

\theoremstyle{remark}
\newtheorem{remark}[thm]{Remark}

\theoremstyle{definition}
\newtheorem{example}[thm]{Example}

\newcommand\ra{\rightarrow}
\newcommand\eps{\varepsilon}

\DeclareMathOperator{\sn}{sn}
\DeclareMathOperator{\abn}{abn}

\title{Embedding Properties in Central Products}

\author{Dandrielle Lewis, Ayah Almousa, and Eric Elert}

\date{\today}

\begin{document}
\maketitle

\begin{abstract}
In this article, we study group theoretical embedding properties of subgroups in central products of finite groups.  Specifically, we give characterizations of normal, subnormal, and abnormal subgroups of a central product of two groups, prove these characterizations, and provide examples.

\end{abstract}

\section{Introduction}
In group theory, studying subgroups of groups is important, but it is not easy to answer the question ``What can we say about the subgroups?'' for products of groups.  Hence, investigating subgroups that satisfy specific properties, known as characterizations, in products of groups is of particular interest.  Edouard Goursat first determined how to identify and describe a subgroup of a direct product of two groups in \cite{goursat1889}, and his result has been refined as can be seen in \cite{brewster12}.  Goursat's work provides the backbone necessary to understand subgroups of a direct product, and many mathematicians have since made significant contributions to this area.
This article presents characterizations of subgroups in a central product of two groups, where a central product is a direct product with an amalgamated center. 

Recently, more work has been done on subgroups of a direct product of two groups.  For example, in \cite{brewster10}, \cite{brewster12},\cite{brewster07}, and \cite{brewster09}, necessary and sufficient conditions have been provided to characterize subgroups of a direct product including normal and subnormal subgroups.  However, very little work has been done to characterize these subgroups for a central product of two groups.  In this article, we begin to characterize subgroups of central products, focusing on normal, subnormal, and abnormal subgroups.  We provide equivalent conditions for normality, subnormality, and abnormality, and give applications of our results.

In Section 2, we define and review basic group theory terms and discuss direct products by introducing notation necessary to understand propositions presented from other work on direct products.  In Section 3, central products are defined internally and externally, and a concrete example of a central product is provided.  In Section 4, we characterize normal subgroups by examining the commutator of subgroups of a central product; we prove this result and provide an application using the example introduced in Section 3.  In Section 5, subnormal subgroups are defined, and a characterization of subnormal subgroups of a central product and an example are provided.  In Section 6, we define abnormal subgroups, determine and prove lemmas concerning abnormal subgroups, and present a characterization of abnormal subgroups in a central product.  Finally, in Section 7, a summary of our results and future work on characterizations of pronormal subgroups of central products are given.

\section{Preliminaries}
In this article, we consider only finite groups. Our notation and terminology are standard, but we define terms and state lemmas that are necessary to understand later information.

Let $G$ be a group. We say $G$ is \textbf{\textit{simple}} if it has no normal subgroups other than itself and the trivial subgroup. For elements $a,b\in G$, we write $a^b=b^{-1}ab$, and if $A$ is a nonempty subset of $G$, $$A^G = \{A^g |\hspace{.1cm} g \in G\}.$$  The \textbf{\textit{commutator}} of $a$ and $b$, $[a,b]$, is given by $a^{-1}a^b$. For $A,B\leq G$, $[A,B]$ will denote the subgroup generated by all the commutators $[a,b]$ where $a\in A$ and $b\in B$. We define $[a_1,a_2,a_3]=[[a_1,a_2],a_3]$ and define commutators of $n$ elements recursively where $[a_1,a_2,\ldots,a_n] = [[a_1, a_2, ... a_{n-1}], a_n]$ for $a_i$ in a group $G$, for all $i \in \mathbb{N}, i\leq n$. The \textbf{\textit{center}} of a group $G$ is $Z(G)=\{g\in G|\hspace{.1cm} gx=xg\text{, } \forall x\in G\}$. 
The \textbf{\textit{derived series}} of a group $G$ is the series
$$
G=G^{(0)}\geq G^{(1)}\geq\cdots
$$
where $G^{(i)}=[G^{(i-1)},G^{(i-1)}]$. A group $G$ is \textbf{\emph{solvable}} if $G^{(i)}=1$ for some $i\geq 1$.

The following lemma contains classical results about the commutator subgroup $[A,B]$. This lemma is presented here because it is useful in later proofs.

\begin{lemma}\label{commproperties} Let $A,B,C$ be subgroups of a group $G$.
\begin{enumerate}
\item $[A,B]=[B,A]\unlhd \langle A,B\rangle$.
\item $[A,B]\leq A$ if and only if $B\leq N_G(A)$, where $N_G(A)$ is the normalizer of $A$ in $G$.
\item If $\alpha: G\ra \alpha(G)$ is a group homomorphism, then $\alpha([A,B])=[\alpha(A),\alpha(B)]$.
\item If $A$ and $B$ are normal subgroups of $G$, then $[A,B]$ is a normal subgroup of $G$.
\item $\langle A^B\rangle=A[A,B]$.
\end{enumerate}
\end{lemma}

\begin{proof}
See \cite{dh92}.
\end{proof}

Whenever we discuss direct products, they will be viewed externally unless otherwise stated. Let $U_1$ and $U_2$ be groups and consider the direct product $U_1\times U_2$. The maps $\pi_i: U_1\times U_2\rightarrow U_i$ given by $\pi_i((u_1,u_2))=u_i$ for $i=1,2$ are standard projections. $\overline{U}_1=\{(u_1,1)| \hspace{.1cm} u_1\in U_1\}$ and $\overline{U}_2=\{(1,u_2)| \hspace{.1cm} u_2\in U_2\}$, where $\overline{U}_i$ is a subgroup of $U_1\times U_2$ for $i=1,2$.

The following lemma is utilized throughout this article, and it provides us with a one-to-one correspondence between normal subgroups. This result can be found in abstract algebra textbooks such as \cite{dummitfoote} and \cite{hungerford} and follows from the classical Lattice Isomorphism Theorem.

\begin{lemma}\label{correspondence}
If $\varphi: G\ra H$ is an epimorphism of groups and $S_\varphi(G)=\{K\leq G | \ker\varphi\subseteq K\}$ and $S(H)$ is the set of all subgroups of $H$, then the assignment $K\mapsto \varphi(K)$ is a one-to-one correspondence between $S_\varphi(G)$ and $S(H)$. Under this correspondence, normal subgroups correspond to normal subgroups.
\end{lemma}

\begin{proof} See \cite{hungerford}.
\end{proof}

\begin{remark}
Let $\phi: G\ra H$ be a group epimorphism, and let $W\leq H$. For the remainder of this article, under such an epimorphism $\phi$, we will assume that $\ker\phi\leq\phi^{-1}(W)$, where $\phi^{-1}(W)$ is the preimage of $W$.
\end{remark}

\section{Central Products}
Central products have historically been useful for the characterization of extraspecial groups.  Let $p$ be a prime.  Extraspecial groups are $p-$groups where the center, commutator subgroup, and the intersection of all the maximal subgroups have order $p$.  All extraspecial groups of order $p^{2n+1}$ for some $n>0$ and $p=2$ or $p\neq2$ can always be written as a central product of extraspecial groups of order $p^3$ \cite{lgm02}. When discussing central products, we will adopt the definition and notation presented in \cite{dh92}:

\begin{definition}\label{cpdef} 
Let $U_1,U_2\leq G$. Then $G$ is an internal central product of $U_1$ and $U_2$ if
\begin{enumerate}
\item $G=U_1U_2$, and
\item $[U_1,U_2]=1$.
\end{enumerate}
\end{definition}

This definition is standard as can be seen in \cite{lgm02}, and observe that it implies that $U_i \unlhd G$ for $i = 1,2$ and $U_1 \cap U_2 \leq Z(U_1) \cap Z(U_2)$. The following lemma establishes the relationship between central products and direct products and serves as a powerful tool for proving results.

\begin{lemma}\label{cpepi}\cite{dh92}
Let $G$ be a group such that $U_1,U_2\leq G$, $D=U_1\times U_2$ is the external direct product, $\overline{U}_1=\{(u_1,1)| \hspace{.1cm} u_1\in U_1\}$ and $\overline{U}_2=\{(1,u_2)|\hspace{.1cm} u_2\in U_2\}$. Then the following are equivalent:
\begin {enumerate}
\item $G$ is an internal central product of $U_1$ and $U_2$.
\item There exists an epimorphism $\varepsilon : D\rightarrow G$ such that $\varepsilon(\overline{U}_1)=U_1$ and $\varepsilon(\overline{U}_2)=U_2$.
\end{enumerate}
\end{lemma}

\begin{remark}
Throughout this article we will view central products internally wherever it is sensible. When central products are discussed externally, we will do so as is described in the following construction:
\end{remark}

\begin{thm}\label{externalcp}\cite{dh92}
Let $V_1, V_2$ be finite groups, and assume that $A$ is an abelian group for which there exists monomorphisms $\mu_i: A\rightarrow Z(V_i)$ for $i=1,2$.  Let $D$ denote the external direct product $V_1\times V_2$, let $\overline{V}_1=\{(v_1,1)|\hspace{.1cm} v_1\in V_1\}$, and  $\overline{V}_2=\{(1,v_2)|\hspace{.1cm} v_2\in V_2\}$.   Then $\overline{V}_i\cong V_i$. Set
$$
N=\{(\mu_1(a),\mu_2(a^{-1}))|\hspace{.1cm} a\in A\}
$$
Then $N\unlhd D$, $\overline{V}_i\cap N=1$, and with $U_i=\overline{V}_iN/N$, the quotient group $G=D/N$ has the following properties:
\begin{enumerate}
\item $G$ is a central product of the subgroups $U_1$ and $U_2$, and $U_i\cong V_i$ for $i=1,2$.
\item $U_1\cap U_2=A_iN/N\cong A$, where
\begin{align*}
A_1=\{(\mu_1(a),1)| \hspace{.1cm} a\in A\}\\
A_2=\{(1, \mu_2(a))| \hspace{.1cm} a\in A\}
\end{align*}
\end{enumerate} 
\end{thm}

The external central product $D/N$ as constructed above is isomorphic to the internal central product $G = U_1U_2$ where $U_i \cong V_i$ for $i=1,2$.  We now show that $\phi: D/N\ra G$ given by $\phi((a,b)N)=ab$ is well defined but leave it to the reader to prove $\phi$ an isomorphism. 

\textit{Well defined:}  Let $(u_1,u_2)N=(v_1,v_2)N.$ We wish to show that $\phi((u_1,u_2)N)=\phi((v_1,v_2)N).$  By assumption, $(v_1,v_2)^{-1}(u_1,u_2)N=N$.  Hence, $(v_1,v_2)^{-1}(u_1,u_2)=(v_1^{-1}u_1,v_2^{-1}u_2)\in N$.  Because the identity maps to the identity,
$$
\phi((v_1^{-1}u_1,v_2^{-1}u_2)N)=v_1^{-1}u_1v_2^{-1}u_2=1$$
Now, $$
\phi((u_1,u_2)N)=u_1u_2=(v_1v_1^{-1})u_1(v_2v_2^{-1})u_2
$$
However, elements of $U_2$ commute with elements of $U_1$, which yields
$$
v_1v_2(v_1^{-1}u_1v_2^{-1}u_2)=v_1v_2=\phi((v_1,v_2)N).
$$
Therefore, $\phi$ is well defined. 

Internal central products are unique because $G = U_1U_2$ is determined by the choice of subgroups $U_1$ and $U_2$ of $G$. However, external central products are not unique because the construction depends on the choice of monomorphisms $\mu_1$ and $\mu_2$.

Let us develop an intuition for the construction of an external central product via an example that we will return to in later sections.

\begin{example}\label{d8c4}
Consider the dihedral group of order eight, $D_8$, and the cyclic group of order four, $C_4$, with the following respective group presentations:
\begin{align*}
&D_8=\langle r,s| \hspace{.1cm} r^4=s^2=1, rsr=s\rangle\\
&C_4=\langle y|\hspace{.1cm}  y^4=1\rangle
\end{align*}
We use \ref{externalcp} to construct $D_8 \circ C_4$, a central product of $D_8$ and $C_4$, which is a group of order 16.

Let $U_1=D_8, U_2=C_4$. Fix $A=Z(D_8)\cong C_2$. Then we define the monomorphisms
\begin{align*}
\mu_1: Z(D_8)\ra Z(D_8)\\
\mu_2: Z(D_8)\ra Z(C_4)
\end{align*}
by $\mu_1$ as the identity map and $\mu_2(r^2)=y^2$. Then $N=\{(\mu_1(a),\mu_2(a^{-1}))|\hspace{.1cm} a\in A\}=\{(1,1),(r^2,y^2)\}$, and
$$
D_8\circ C_4=\frac{D_8\times C_4}{N}.
$$

Without loss of generality, suppose $\overline{r} = (r,1)N$, $\overline{s} = (s,1)N$, and $\overline{y} = (1,y)N$. Then

$$
D_8\circ C_4=\langle \overline{r},\overline{s},\overline{y}| \hspace{.1cm} \overline{r}^4=\overline{s}^2=N, \hspace{.1cm} \overline{r}\overline{s}\overline{r}=\overline{s}, \hspace{.1cm} \overline{r}^2=\overline{y}^2,\hspace{.1cm} \overline{r}\overline{y}=\overline{y}\overline{r}, \hspace{.1cm} \overline{s}\overline{y}=\overline{y}\overline{s}\rangle.
$$

Note that because of the isomorphism between the internal and external presentation of a central product, we know that we may also write this group as

$$
D_8C_4=\langle r,s,y,| \hspace{.1cm} r^4=s^2=1, \hspace{.1cm} rsr=s,\hspace{.1cm} r^2=y^2, \hspace{.1cm} ry=yr, \hspace{.1cm} sy=ys\rangle.
$$

\end{example}

\begin{remark} All results characterizing subgroups of central products in the remaining sections will be stated for internal central products. It is important to note that all results will apply to external central products due to the isomorphism between internal and external central products.

\end{remark}

\section{Normal Subgroups of Central Products}

The goal of this section is to characterize normal subgroups of a central product of two groups. \ref{cpepi} allows us to define the central products of groups $U_1$ and $U_2$ by an epimorphism from $D = U_1 \times U_2$ onto $G = U_1U_2$ with $\eps(\overline{U}_i) = U_i$ for $i = 1,2$. We begin with a lemma from \cite{brewster07} that characterizes normal subgroups of a direct product and serves as the inspiration for our characterization.

\begin{lemma}\label{dpnormal}
Let $U_1, U_2$ be groups. Then $N\unlhd U_1\times U_2$ if and only if $[N,\overline{U_i}]\leq N\cap\overline{U_i}$ where $i = 1,2$.
\end{lemma}

In \ref{cpepi}, an epimorphism between direct products and central products was established to give an alternate and more useful way of viewing central products. Given an epimorphism, we know by \ref{correspondence} that there is a one-to-one correspondence between normal subgroups of a domain which contain the kernel and normal subgroups of the codomain. This serves as a motivation for the following result.

\begin{proposition}\label{cpnormal}
Let $G = U_1U_2$ be the central product of subgroups $U_1$ and $U_2$ defined by the epimorphism $\eps: D\ra G$, where $D=U_1\times U_2$. Let $H\leq G$. Then $H\unlhd G$ if and only if $\eps^{-1}(H)\unlhd D$.
\end{proposition}

\begin{proof}
By \ref{cpepi}, an epimorphism $\eps: D\ra G$ exists. Since there is a one-to-one correspondence between the set $S_\eps(D)=\{K\leq D| \hspace{.1cm} \ker\eps\leq K\}$ and $S(G)$, the set of all subgroups of $G$, we know for any normal subgroup of $D$ containing $\ker\eps$, its image is a normal subgroup of $G$. That is,
$$
H\unlhd G \iff \ker\eps\leq\eps^{-1}(H)\unlhd D
$$
as desired.
\end{proof}

The following diagram demonstrates the relationship established in \ref{cpnormal} between normal subgroups of $U_1U_2$ and the normal subgroups of $U_1\times U_2$ containing the $\ker\eps$:



\begin{figure}[h]
$$\xymatrix{
{U_1 \times U_2} \ar@{-}[d]^{\trianglelefteq} \ar[r]^{\varepsilon} & {U_1U_2} \ar@{-}[d]^{\trianglelefteq} \\  
{\varepsilon^{-1}(H)} \ar@{-}[d]   & {H}\\
{\ker\varepsilon}
}
$$
\caption{Diagram for \ref{cpnormal}}
\end{figure}


To further characterize normal subgroups of the central product $U_1U_2$, we generalize \ref{dpnormal}. This generalization provides an efficient way to find normal subgroups of $U_1U_2$ and can be easily implemented in computer algebra systems such as GAP or Sage. 

\begin{proposition}\label{cpcom}
Let $G = U_1U_2$ be the central product of subgroups $U_1$ and $U_2$. Then $H\unlhd G$ if and only if $[U_i, H]\leq U_i\cap H$ where $i = 1,2$.
\end{proposition}


\begin{proof}
If $H\unlhd G$, then $N_G(H)=G$ and $U_i\leq N_G(H)$.  So $[U_i, H]\leq H$ by \ref{commproperties}. Similarly, because $U_i\unlhd G$, we know $[U_i, H]\leq U_i$.  Therefore, $[U_i, H]\leq U_i\cap H$ for $i=1,2$ as desired.

Assume that $[U_i, H]\leq U_i\cap H$ for $i=1,2$. Let $h\in H$ and $g\in G$, and consider $h^g$. Because $G$ is an internal central product, we can write $g=u_1u_2$ for some $u_1\in U_1$ and $u_2\in U_2$. Therefore,

$$
h^g=h^{u_1u_2}=(u_1u_2)^{-1}h(u_1u_2)=u_2^{-1}u_1^{-1}hu_1u_2=u_2^{-1}(hh^{-1})u_1^{-1}hu_1u_2.
$$
Note that $h^{-1}u_1^{-1}hu_1\in H\cap U_1$ by assumption.  Let  $h_1= h^{-1}h^{u_1}$ and $h_2 = hh_1$.  Then $$h^g = h_2^{u_2} = h_2h_2^{-1}h_2^{u_2}.$$  Since $h_2^{-1}h_2^{u_2} \in U_2 \cap H$, it follows that $h^g \in H$ and $H\unlhd G$.

\end{proof}

In the following diagram we give a visual representation of \ref{cpcom}.


\begin{figure}[h]
\begin{center}
$$\xymatrix{
 & {U_1U_2}\ar@{-}[dl] \ar@{-}[dr]^{\trianglelefteq}\\
{U_i} \ar@{-}[dr] & & {H} \ar@{-}[dl] \\
&{U_i \cap H} \ar@{-}[d] \\
&{[U_i,H]} 
}$$
\caption{Diagram for \ref{cpcom}. Observe that we must have $H\lhd U_1U_2$ for accuracy; otherwise, the commutator subgroup $[U_i, H]$ would not be a subgroup of $U_i\cap H$ for $i=1,2$. Lines without arrowheads indicate subgroup containment with smaller groups below the groups in which they are contained.}
\end{center}
\end{figure}

The following theorem provides a summary of our results shown in this section.

\begin{thm}\label{tfaenormal}
Let $G = U_1U_2$ be the central product of subgroups $U_1$ and $U_2$ defined by the epimorphism $\eps: D \rightarrow G$ where $D = U_1 \times U_2$. Let $H \leq G$. Then the following are equivalent:
\begin{enumerate}
\item $H\unlhd G$.
\item $\eps^{-1}(H)\unlhd D$.
\item $[\overline{U}_i, \eps^{-1}(H)]\leq \overline{U}_i\cap \eps^{-1}(H)$ for $i=1,2$.
\item $[U_i, H]\leq H\cap U_i$ for $i=1,2$.
\end{enumerate}
\end{thm}


Note that $(ii) \iff (iii)$ is a direct application of \ref{dpnormal}.

\begin{example}
To illustrate the usefulness of \ref{tfaenormal}, consider $D_8C_4$ from \ref{d8c4}. In this example, we determined its subgroups, their respective orders, their isomorphism types, and the normal subgroups using our results.  This information is provided in the table below.

\begin{center}
\begin{tabular}{|l|c|l|c|l|}
\hline
Isomorphism Type & $\#$ & Subgroups& $\#$ Normal & Normal Subgroups\\ \hline
Trivial group & 1 & $\{1\}$ & 1 & All\\
$C_2$ & 7 & $\langle r^2\rangle, \langle s\rangle, \langle rs\rangle, \langle r^2s\rangle, \langle r^3s\rangle, \langle ry\rangle, \langle r^3y\rangle$ & 1 & $\langle r^2\rangle$\\
$C_4$ & 4 & $\langle y\rangle, \langle r\rangle, \langle sy\rangle, \langle rsy\rangle$ & 4 & All\\
$V_4$ & 4 & $\langle r^2,s\rangle, \langle r^2,rs\rangle, \langle r^2,sy\rangle$ & 4 & All\\
$Q_8$ & 1 & $\langle r, sy\rangle$ & 1 & All\\
$C_4\times C_2$ & 3 & $\langle r,y\rangle, \langle s,y\rangle, \langle rs,y\rangle$ & 3 & All\\
$D_8$ & 3 & $\langle r,s\rangle, \langle ry,sy\rangle, \langle rsy,ry\rangle$ & 3 & All\\
$D_8C_4$ & 1 & $D_8C_4$ & 1 & All\\
\hline
\hline
Total & 23 & & 17 &\\ \hline
\end{tabular}
\end{center}


We provide a sample calculation using \ref{tfaenormal}. Consider the subgroup $\langle rsy\rangle \leq D_8C_4$.  To show this subgroup is normal, it suffices to verify that $[\langle rsy\rangle,D_8]\leq \langle rsy\rangle\cap D_8$ and $[\langle rsy\rangle, C_4]\leq \langle rsy\rangle \cap C_4$. It is sufficient to check this using the generators of each group:

\begin{align*}
&[rsy, y]=1\leq \langle rsy\rangle \cap C_4\\
&[rsy,s]=r^3sysrsys=y^2=r^2\\
&[rsy,r]=r^3syr^3rsyr=r^2\\
&[\langle rsy\rangle,D_8]=\langle r^2\rangle\leq\langle rsy\rangle\cap D_8
\end{align*}

Hence, $\langle rsy\rangle \unlhd D_8C_4$.
Our result is also helpful to identify subgroups that are not normal.  For example, to see the subgroup $\langle s\rangle\leq D_8C_4$ is not normal, we verify that $[\langle s\rangle,D_8]\nleq\langle s\rangle\cap D_8$.
$$
[s,r]=sr^3sr=r^2\notin \langle s\rangle.
$$
So, $\langle s\rangle$ is not a normal subgroup of $D_8C_4$.

\end{example}

\section{Subnormal Subgroups of Central Products}

Our next goal is to characterize subnormal subgroups of central products. We use the following definition of subnormal from \cite{dh92} with adjusted notation.
\begin{definition}
Let $H\leq G$. We call $H$ \textbf{subnormal} in $G$, written $H\sn G$, if there exists a chain of subgroups $H_0,H_1,\ldots,H_r$ such that
$$
H=H_0\unlhd H_1\unlhd \ldots \unlhd H_{r-1}\unlhd H_r=G
$$
This is called a \textbf{subnormal chain} from $H$ to $G$. If such a chain is the minimal possible chain from $H$ to $G$, we say that the subnormal subgroup is of \textbf{defect} $r$.
\end{definition}

Subnormal subgroups were completely characterized in direct products by Hauck in \cite{hauck}.  In the following theorems, we adopt, with some modifications, the notation for this classification as it is presented in \cite{brewster07}.

\begin{lemma}\label{dpsubnormal}
A subgroup $K$ of $U_1\times U_2$ is subnormal in $U_1\times U_2$ (of defect $r$) if and only if $$[U_i,\pi_i(K),\overset{r}{\ldots},\pi_i(K)]\leq K\cap U_i$$ for $i=1,2$.
\end{lemma}

In order to classify subnormal subgroups of central products, we seek a correspondence theorem similar to \ref{correspondence} in order to establish a relationship between subnormal subgroups of the direct product of two groups and subnormal subgroups of a central product of two groups. We do so using the following proposition.

\begin{proposition}\label{sncorrespondence}
Let $f:G\ra H$ be an epimorphism. Let $S_f(G)=\{K\leq G| \hspace{.1cm} \ker f\leq K$\}, and let $S(H)$ be the set of all subgroups of $H$. Then there is a one-to-one correspondence $K\mapsto f(K)$ between $S_f(G)$ and $S(H)$, where subnormal subgroups correspond to subnormal subgroups.
\end{proposition}

\begin{proof}
To show that subnormal subgroups correspond to subnormal subgroups under the bijection between $S_f(G)$ and $S(H)$, it suffices to show that for $U,V\in S_f(G)$, $V\unlhd U$ if and only if $f(V)\unlhd f(U)$.
To prove this, we construct an epimorphism $\phi: U\ra f(U)$ given by $\phi(u)=f(u)$ for all $u\in U$.

\textit{$\phi$ is a homomorphism:} For any $u,v\in U$, $\phi(uv)=f(uv)=f(u)f(v)=\phi(u)\phi(v)$ because $f$ is a homomorphism.

\textit{$\phi$ is surjective:} Let $x\in f(U)$. Then $x=f(u)$ for some $u\in U$, but $\phi(u)=f(u)=x$.  So $\phi$ is onto and is therefore an epimorphism.

Now it suffices to show that $\ker\phi=\ker f$. Let $a\in\ker\phi$. Then $\phi(a)=f(a)=1$, and $\ker \phi\subseteq\ker f$. Similarly, because $\ker f\leq U$, we have that if $b\in\ker f$, then $f(b)=\phi(b)=1$.  Hence, $\ker f \subseteq\ker\phi$ and we have equality, as desired.

By \ref{correspondence}, we know that there exists a one-to-one correspondence between subgroups of $U$ containing $\ker\phi=\ker f$ and subgroups $f(U)$ such that normal subgroups correspond to normal subgroups. Then $\ker f\leq V\unlhd U\leq G$ if and only if $f(V)\unlhd f(U)$. 

Applying this result to a chain of normal subgroups will prove the desired theorem. Specifically, for $\ker f\leq A$, we get $A=A_0\unlhd A_1\unlhd\ldots\unlhd A_r=G$ if and only if $f(A_0)\unlhd f(A_1)\unlhd\ldots\unlhd f(A_r)=f(G)=H$.  Therefore, $f(A)\sn H$.
\end{proof}

Utilizing \ref{sncorrespondence} and \ref{cpepi}, we obtain the following characterization of subnormal subgroups.

\begin{proposition}\label{cpsubnormal} Let $G=U_1U_2$ be the central product of subgroups $U_1$ and $U_2$ defined by the epimorphism $\eps: D\ra G$ where $D=U_1\times U_2$. Then $H\sn G$ if and only if $\eps^{-1}(H)\sn D$.
\end{proposition}

\begin{proof}
By \ref{cpepi}, we know that for any central product $G$ there exists an epimorphism $\eps:D\ra G$ such that $\eps(\overline{U}_i)=U_i$. By \ref{sncorrespondence}, we know that $\eps$ gives rise to a one-to-one correspondence between the set of subnormal subgroups of $D$ which contain $\ker\eps$ and the set of subnormal subgroups of $G$. Therefore, $H\sn G$ if and only $\eps^{-1}(H)\sn D$ as desired.
\end{proof}

Naturally, our next goal was to answer the question: ``Could we further characterize subnormal subgroups in a way that would be useful for computations?''  In the following result, we provide a characterization of subnormal subgroups for central products which could be implemented in GAP or Sage.


\begin{proposition}\label{sncommutator}
Let $G=U_1U_2$ be the central product of subgroups $U_1$ and $U_2$. Then $H\sn G$ of defect $r$ if and only if $[U_i,H,\overset{r}{\ldots},H]\leq U_i\cap H.$
\end{proposition}


\begin{figure}[h]
\begin{center}
$$\xymatrix{
 & {U_1U_2}\ar@{-}[dl] \ar@{-}[dr]^{\sn}\\
{U_i} \ar@{-}[dr] & & {H} \ar@{-}[dl] \\
&{U_i \cap H} \ar@{-}[d] \\
&{[U_i,H,\overset{r}{\ldots},H]} 
}$$
\caption{Observe that we must have $H\sn U_1U_2$ for accuracy; otherwise, the commutator subgroup $[U_i, H,\overset{r}{\ldots},H]$ would not be a subgroup of $U_i\cap H$ for $i=1,2$. Lines without arrowheads indicate subgroup containment with smaller groups below the groups in which they are contained.}
\end{center}
\end{figure}



The following theorem provides a summary of our results shown in this section.

\begin{thm}\label{tfaesn}
Let $G = U_1U_2$ be the central product of subgroups $U_1$ and $U_2$ defined by the epimorphism $\eps: D \rightarrow G$ where $D=U_1\times U_2$. Let $H\leq G$. Then the following are equivalent:
\begin{enumerate}
\item $H\sn G$ of defect less than or equal to $r$.
\item $\eps^{-1}(H)\sn D$ of defect less than or equal to $r$. 
\item $[\overline{U}_i,\eps^{-1}(H),\overset{r}{\ldots},\eps^{-1}(H)]\leq \overline{U}_i\cap \eps^{-1}(H)$.
\item $[U_i,H,\overset{r}{\ldots},H]\leq U_i\cap H$.
\end{enumerate}
\end{thm}

\begin{proof}
For ease of notation set $K = \eps^{-1}(H)$.\\
For $(i)\iff (ii)$ and $(ii)\iff (iii)$ see \ref{cpsubnormal} and \ref{dpsubnormal} respectively.\\
$(iii)\Rightarrow (iv)$: Suppose $[\overline{U}_i,K,\overset{r}{\ldots},K]\leq \overline{U}_i\cap K$. Because $\eps$ is a homomorphism, we know that

$$
[\overline{U}_i,K,\overset{r}{\ldots},K]\leq \overline{U}_i\cap K \hspace{.2cm} \text{    implies    } \hspace{.2cm}  
\eps([\overline{U}_i,K,\overset{r}{\ldots},K]) \leq \eps(\overline{U}_i\cap K).
$$

Additionally, since $\eps$ is onto, $\eps(\overline{U}_i\cap K)\leq \eps(\overline{U}_i)\cap\eps(K)$.
Therefore,
$$
[U_i,H,\overset{r}{\ldots},H]=[\eps(\overline{U}_i),\eps(K),\overset{r}{\ldots},\eps(K)]=\eps[\overline{U}_i, K, \overset{r}{\ldots},K]\leq \eps(\overline{U}_i\cap K) \leq \eps(\overline{U}_i)\cap \eps(K)=U_i\cap H
$$
as desired.

$(iv)\Rightarrow (iii)$: Suppose $[U_i, H, \overset{r}{\ldots},H]\leq U \cap H$. Now $\eps^{-1}([U_i, H, \overset{r}{\ldots},H]) \leq \eps^{-1}(U_i\cap H)$.

Since $\overline{U}_i\leq \eps^{-1}(U_i)$, we know  [$\overline{U}_i,\eps^{-1}(H),\ldots,\eps^{-1}(H)]\leq [\eps^{-1}(U_i),\eps^{-1}(H),\ldots,\eps^{-1}(H)]$ and $\overline{U}_i\cap\eps^{-1}(H)\leq \eps^{-1}(U_i)\cap\eps^{-1}(H)$.

Without loss of generality, consider $\overline{U}_1$. Let $[(u,1),(h_1,h_2)]$ be an arbitrary element of $[\overline{U}_1,\eps^{-1}(H)]$. Then $[(u,1),(h_1,h_2)]=(u,1)^{-1}(h_1,h_2)^{-1}(u,1)(h_1,h_2)=(u^{-1}h_1^{-1}uh_1,h_2^{-1}h_2)=(u^{-1}h_1^{-1}uh_1,1)\leq \overline{U}_1$.

Now if $(d_1,d_2)\in D$ and $(u,1)\in\overline{U}_1$, then $[(u,1),(d_1,d_2)]=(u^{-1}h_1^{-1}uh_1,1)\leq \overline{U}_1$.

By induction, $[\overline{U}_i,\eps^{-1}(H),\ldots,\eps^{-1}(H)]\leq \overline{U}_i$.  Since $[\overline{U}_i,\eps^{-1}(H),\ldots,\eps^{-1}(H)]\leq\eps^{-1}(U_i)\cap\eps^{-1}(H)$, we have

$$
[\overline{U}_i,\eps^{-1}(H),\ldots,\eps^{-1}(H)]\leq \overline{U}_i\cap \eps^{-1}(U_i)
$$






\end{proof}

\begin{remark}
If we defined multi-fold commutators as $[a_1,a_2,\ldots,a_n] = [a_1, [a_2, ... ,a_n]]$ for $a_i$ in a group $G$ instead of $[[a_1, a_2, ... ,a_{n-1}], a_n]$, our result in \ref{sncommutator} would change to $H\leq G$ is subnormal in $G$ of defect $r$ if and only if $[H,\overset{r}{\ldots},H, U_i]\leq U_i\cap H$. Similar for \ref{tfaesn}.
\end{remark}

\begin{example}
To illustrate \ref{tfaesn}, consider $D_8C_4$ from \ref{d8c4}.  Using this example, we found that every subgroup of $D_8C_4$ is subnormal.

Consider the subgroup $\langle s\rangle \leq D_8C_4$. It will be shown that $\langle s\rangle \sn D_8C_4$ with defect 2. Applying \ref{tfaesn}, we verify that $[D_8,\langle s\rangle,\langle s\rangle]\leq D_8\cap \langle s\rangle$ and $[C_4,\langle s\rangle,\langle s\rangle]\leq C_4\cap \langle s\rangle$. It is sufficient to check this using the generators of each group:

\begin{align*}
&[r, s, s] = [[r,s], s] = [r^{-1}s^{-1}rs, s] = [r^3srs,s] = [r^2, s] = r^2sr^2s = 1 \leq D_8\cap\langle s\rangle\\
&[s, s, s] = [[s,s],s] = [1,s] = 1 \leq D_8\cap\langle s\rangle\\
&[y, s, s] = [[y,s], s] = [y^{-1}s^{-1}ys, s] = [y^3sys,s] = [y^4s^2, s] = [1, s] = 1 \leq C_4 \cap \langle s\rangle\\
\end{align*}

Hence, $[D_8,\langle s\rangle,\langle s\rangle]\leq D_8\cap \langle s\rangle$ and $[C_4,\langle s\rangle,\langle s\rangle]\leq C_4\cap \langle s\rangle$ implies $\langle s \rangle \sn D_8C_4$ of defect 2. 

\end{example}

\section{Abnormal Subgroups of Central Products}

Our final goal is to characterize abnormal subgroups of central products of two groups. Abnormal subgroups were originally studied due to their connection to the classification of finite groups because the normalizer of any Sylow subgroup of a group is always abnormal. There are many properties known about abnormal subgroups in finite solvable groups.  Abnormal subgroups and normal subgroups are opposites.  Moreover, a maximal subgroup is abnormal if and only if it is not normal.

To establish properties of abnormal subgroups and how they are viewed in central products, we begin this section by presenting technical lemmas.

\begin{lemma}\label{preimage}
Let $\varepsilon : D \rightarrow G$ be an epimorphism. Let $g = \eps(d)$ and $W\leq G.$  Then $\varepsilon^{-1}(W^g) = (\varepsilon^{-1}(W))^d.$
\end{lemma}

\begin{proof}
Let $x\in \eps^{-1}(W)^d$. Then $x=y^d$ for some $y\in\eps^{-1}(W)$. So, $\eps(y)\in W$ and as a consequence $(\eps(y))^g\in W^g$. Also, $(\eps(y))^g=\eps(y^d)=\eps(x)$
which implies that $\eps(x)\in W^g$.  Therefore, $x\in \eps^{-1}(W^g)$.

Suppose $x\in \eps^{-1}(W^g)$. Then for some $h\in W$, $\eps(x)=h^g=h^{\eps(d)}=(\eps(a))^{\eps(d)}=\eps(a^d)$ for some $a\in D$.
Therefore, $x^{-1}a^d=z$, where $z\in \ker\eps$. That is, $x=a^dz^{-1}\in (\eps^{-1}(W))^d$. 
Therefore, $\eps^{-1}(W^g)=(\eps^{-1}(W))^d$.
\end{proof}

\begin{lemma}\label{joinpreimage}
Let $\varepsilon : D \rightarrow G$ be an epimorphism, $g = \eps(d)$, and $H\leq G.$  Then $\langle\varepsilon^{-1}(H), (\varepsilon^{-1}(H))^d\rangle$ $= \varepsilon^{-1}\langle H, H^g\rangle.$
\end{lemma}

\begin{proof}
Let $x\in\langle H,H^g\rangle$. Then $x=a_1b_2a_2b_2\cdots a_nb_n$ where $a_i\in H, b_i\in H^g$. Since $\eps$ is an epimorphism, there exists a $y\in D$ such that $\eps(y)=x$.  Furthermore, for $a_i, b_i \in G$, there exists $c_i,d_i \in D$ respectively such that $\eps(c_i)=a_i$ and $\eps(d_i)=b_i$.  Then $\eps(y)=x\in\langle H,H^g\rangle$ implies $y \in \eps^{-1}\langle H,H^g\rangle = \langle\eps^{-1}(H),\eps^{-1}(H^g)\rangle = \langle\eps^{-1}(H),\eps^{-1}(H)^d\rangle$ by \ref{preimage}.

Let $x \in \langle\eps^{-1}(H),(\eps^{-1}(H))^d\rangle$. Then $x=k_1\ell_1\cdots k_m\ell_m$, where $k_i\in \eps^{-1}(H),\ell_i\in (\eps^{-1}(H))^d=\eps^{-1}(H^g)$, and $\eps(x)=\eps(k_1)\eps(\ell_1)\cdots\eps(k_m)\eps(\ell_m)$.  Therefore, $\eps(x)\in\langle H,H^g\rangle$, and $x\in \eps^{-1}\langle H,H^g\rangle$.\end{proof}

We present the formal definition of abnormal subgroups from \cite{dh92}. 

\begin{definition}
Consider $H\leq G$. $H$ is \textbf{abnormal} in $G$, written $H\abn G$, if $g\in\langle H,H^{g}\rangle$,  for all $g\in G$.
\end{definition}

Note if $H \abn G$ and $G$ and $G'$ are groups, then the following properties are satisfied.\cite{dh92}
\begin{enumerate}
\item If $H \leq L \leq G$, then $H \abn L$ and $L \abn G$.
\item $N_G(H) = H$.
\item If $\phi : G \rightarrow G'$ is a homomorphism, then $\phi(H)\abn \phi(G)$.
\end{enumerate}

\begin{example}\label{diag}
Let $G$ be a group. The diagonal subgroup of $G\times G$ is defined as $\Delta=\{(g,g)| \hspace{.1cm} g\in G\}$.  Observe that $\Delta\cong G$.

Consider $G = A_5$, the alternating group of degree of five, which is a classic example of a simple, nonsolvable group. By \cite{thevenaz97}, we know $G$ is simple if and only if $\Delta$ is maximal in $G\times G$. Therefore, $\Delta$ is maximal in $A_5\times A_5$.  Because $\Delta$ is maximal and not normal in $A_5\times A_5$,  $\Delta\abn A_5\times A_5$.
\end{example}

In \cite{brewster09}, abnormal subgroups were characterized for direct products of two groups, where one of the direct factors is solvable.  This result is presented in the following lemma. 

\begin{lemma}\label{dpabnormal}
Let $D = U_1 \times U_2$, where either $U_1$ or $U_2$ is solvable. Let $K\leq D$. Then $K \abn D$ if and only if $\pi_i(K) \abn U_i$ and $K = \pi_1(K) \times \pi_2(K)$.
\end{lemma}

Observe that the above lemma applies when either $U_1$ or $U_2$ is solvable. If $U_1$ or $U_2$ is not solvable, the result will not work.  For example, in $A_5 \times A_5$, $\Delta = A_5$ cannot be written as a product of two subgroups. Therefore, one can see that \ref{dpabnormal} does not apply. 

To characterize abnormal subgroups of central products, a correspondence theorem similar to \ref{correspondence} and \ref{sncorrespondence} was needed to establish a relationship between abnormal subgroups of a direct product of two groups and abnormal subgroups of a central product of two groups. The following lemma provides the necessary correspondence.

\begin{lemma}\label{cpabnormal}
Let $G = U_1U_2$ be the central product of subgroups $U_1$ and $U_2$ defined by the epimorphism $\eps: D\ra G$, where $D=U_1\times U_2$. Let $H\leq G$.
Then $H\abn G$ if and only if $\eps^{-1}(H)\abn D$. 
\end{lemma}

\begin{proof}
For ease of notation, define $K=\eps^{-1}(H)$.

If $K\abn D$, then $x\in\langle K,K^x\rangle\text{ }$, $\forall x\in D$. This implies  $\eps(x)\in\langle\eps(K),\eps(K)^{\eps(x)}\rangle$. Because $\eps$ is an epimorphism, $\eps(D)=G$, and without loss of generality, $\eps(x)$ corresponds to an arbitrary element $y$ of $G$. Therefore, $y\in\langle H,H^y\rangle$, $\forall y\in G$ and $H\abn G$.

Suppose $H\abn G$. By definition $g\in \langle H,H^g\rangle$, $\forall g\in G$. Because $\eps$ is surjective, we know there exists some $d\in D$ such that $\eps(d)=g$.  Therefore, $d \in \eps^{-1}\langle H, H^g\rangle$. By \ref{joinpreimage} $\eps^{-1}\langle H, H^g\rangle=\langle K,K^d\rangle$, where $\eps(d) = g$. Since $\eps$ is an epimorphism, $d \in \langle K,K^d\rangle$, $\forall d\in D$. Hence $K\abn D$. 
\end{proof}

To further characterize abnormal subgroups of the central product $U_1U_2$, we generalized \ref{dpabnormal}. This generalization provides an efficient way to find abnormal subgroups of $U_1U_2$ and can be easily implemented in computer algebra systems such as GAP or Sage. 

The following result provides a summary of the lemmas above.

\begin{thm}\label{abntfae}
Let $G = U_1U_2$ be the central product of subgroups $U_1$ and $U_2$ defined by the epimorphism $\eps: D \rightarrow G$ where $D = U_1 \times U_2$. Let $H \leq G$. Suppose $U_1$ or $U_2$ is solvable. Then the following are equivalent:
\begin{enumerate}
\item $H\abn G$.
\item $\eps^{-1}(H)\abn D$.
\item $\pi_i(K) \abn U_i$ and $K = \pi_1(K) \times \pi_2(K)$, where $K = \eps^{-1}(H)$.
\end{enumerate}
\end{thm}


\begin{thm}\label{abniff}
Let $G = U_1U_2$ be the central product of subgroups $U_1$ and $U_2$ defined by the epimorphism $\eps: D\ra G$, where $D=U_1\times U_2$ and either $U_1$ or $U_2$ is solvable. If $H=V_1V_2$ is a subgroup of $G$ such that $V_i\leq U_i$ and $\eps^{-1}(H)=V_1\times V_2$ contains $\ker\eps$, then $H \abn G$ if and only if $V_i \abn U_i$, for $i = 1,2$. 
\end{thm}

\begin{proof}
$``\Rightarrow"$:  From \ref{cpabnormal} we know $H\abn G$ if and only if $\ker\eps\leq\eps^{-1}(H)\abn D$.
By \ref{dpabnormal}, $\eps^{-1}(H)\abn D$ if and only if $\pi_i(\eps^{-1}(H))\abn U_i$ and $\eps^{-1}(H)=\pi_1(\eps^{-1}(H))\times\pi_2(\eps^{-1}(H))$.

Let $V_i=\pi_i(\eps^{-1}(H))$. Consider the restriction of $\eps$ to $\eps^{-1}(H)$ given by:

$$
\eps: V_1\times V_2\ra \eps(V_1\times V_2)
$$

It is left to the reader to verify that $\eps(\overline{V}_i)=V_i$.  Hence, by \ref{cpepi} so $\eps(V_1\times V_2)$ is an internal central product of $V_1$ and $V_2$ and $\eps(V_1\times V_2)=V_1V_2$.

But note that $V_1\times V_2=\eps^{-1}(H)$ and $\eps(\eps^{-1}(H))=H$. Therefore, $H=V_1V_2$ where $V_i\abn U_i$, for $i =1,2$.

$``\Leftarrow"$: Suppose $H=V_1V_2$ where $V_i\leq U_i$, for $i=1,2$. By assumption, $\eps^{-1}(H)=V_1\times V_2$ and $V_i \abn U_i$, for $i = 1,2$. By \ref{dpabnormal}, we know  $V_1\times V_2\abn U_1\times U_2$ if and only if $V_i\abn U_i$. Finally, since $H\abn G$ if and only if $\eps^{-1}(H)\abn D$ by \ref{cpabnormal}, $H\abn G$, as desired.

\end{proof}









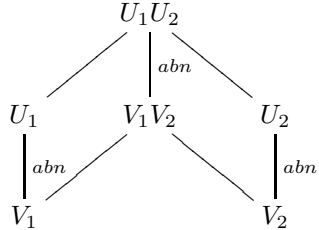
\begin{figure}[h]
\begin{center}
$$\xymatrix{
 & {U_1U_2}\ar@{-}[d]^{abn} \ar@{-}[dl] \ar@{-}[dr]\\
{U_1} \ar@{-}[d]^{abn} & {V_1V_2} \ar@{-}[dr] \ar@{-}[dl] & 
{U_2} \ar@{-}[d]^{abn} \\
{V_1} & & {V_2}
}$$
\caption{Note $H = V_1V_2$ must be abnormal in $G$ for accuracy. As shown, $H$ must be able to be written as a central product of $V_1$ and $V_2$ where $V_i$ is an abnormal subgroup $U_i$ for $i = 1,2$. Lines without arrowheads indicate subgroup containment with smaller groups below the groups in which they are contained.}
\end{center}
\end{figure}

Note for this result about abnormal subgroups of central products of two groups, solvability of one of the central factors is still required. 

\section{Future Work}
While characterizing normal, subnormal, and abnormal subgroups in a central product of groups, many other open questions did arise. One open question is to characterize pronormal subgroups of central products. Results on pronormal subgroups of direct products, as in \cite{brewster09}, may be especially useful and inspirational for this direction of work.

Additionally, we would like to explore the characterization of abnormal subgroups of direct products and central products of two non-solvable groups.

\section{Acknowledgments}
We would like to thank the University of Wisconsin-Stout REU for providing us with the opportunity to pursue this project, and the NSF for funding this project via NSF Grant DMS-1062403.


\end{document}